\documentclass[]{fundam}

\usepackage{amsmath,amssymb}
\usepackage{braket}
\usepackage{tikz}
\usepackage{multirow}
\usetikzlibrary{arrows,automata}

\newcommand{\R}{\mathbb{R}}
\newcommand{\Z}{\mathbb{Z}}
\newcommand{\N}{\mathbb{N}}
\newcommand{\I}{\mathbb{X}}
\newcommand{\B}{\mathcal{B}}
\newcommand{\X}{\mathcal{X}}
\newcommand{\W}{\mathcal{W}}
\newcommand{\INF}{{}^\infty}

\newtheorem{question}{Question}
\newtheorem{conjecture}{Conjecture}

\begin{document}
\title{Playing with Subshifts\thanks{Research supported by the Academy of Finland Grant 131558}}

\author{Ville Salo\\
TUCS -- Turku Centre for Computer Science, Finland, \\
University of Turku, Finland, \\
vosalo{@}utu.fi
\and Ilkka T\"orm\"a\\
TUCS -- Turku Centre for Computer Science, Finland, \\
University of Turku, Finland, \\
iatorm{@}utu.fi}

\maketitle

\runninghead{V. Salo, I. T\"orm\"a}{Playing with Subshifts}

\begin{abstract}
We study the class of word-building games, where two players pick letters from a finite alphabet to construct a finite or infinite word. The outcome is determined by whether the resulting word lies in a prescribed set (a win for player $A$) or not (a win for player $B$). We focus on symbolic dynamical games, where the target set is a subshift. We investigate the relation between the target subshift and the set of turn orders for which $A$ has a winning strategy.
\end{abstract}


\section{Introduction}

Subshifts are the central objects of symbolic dynamics, and also have an important role in coding and information theory. A subshift is a set of infinite sequences over a finite alphabet defined by forbidden patterns. A sequence lies in the subshift if and only if no forbidden pattern occurs in it. Subshifts can be viewed as streams of information flowing from a transmitter towards a receiver, and the forbidden patterns represent restrictions of the medium used. The asymptotic rate at which information can be sent is captured by the notion of entropy.

Consider the following scenario: A person (we will call her Alice, or $A$) is trying to send some information through the medium. She does this by choosing, one by one, a letter from the alphabet, and constructing an infinite sequence from them. She can pick any letters she fancies, as long as the resulting sequence contains no forbidden pattern. But at some prescribed moments an adversary (whom we will call Bob, or $B$) manages to insert a letter to the end of the finite sequence $A$ has constructed thus far. This may cause a forbidden word to appear, destroying the information channel, but not necessarily: if $A$ knows beforehand the moment at which $B$ will interfere, she can choose her sequence so that no matter which letter $B$ inserts, no forbidden pattern results. The situation can be thought of as a game, where $A$ wins if no forbidden pattern ever occurs.

Games of this form have previously been studied in \cite{Pe11}, from the point of view of combinatorics on words. The article in question concentrates on subshifts in which no long enough approximate squares, or patterns approximately of the form $ww$, occur. The results of the paper state that for certain definitions of `approximate' and `long enough', $A$ has a winning strategy even if every choice for $A$ is followed by some $t$ choices for $B$. A similar notion has also appeared in \cite{LoMaPa13}, where the authors study the set of coordinates of a given configuration of a subshift that can be chosen independently of the rest, and the notion of \emph{independence entropy} arising from these considerations. Their methods are related to our combinatorial arguments in Section~\ref{sec:Entropy}.

In this article, we take a different view on the situation. Instead of asking whether a given subshift admits a winning strategy for $A$ when a prescribed order for the turns of $A$ and $B$ is used, we fix a subshift $X$ and study the set of turn orders for which $A$ has a winning strategy on it. We call this set the winning shift of $X$, since it actually turns out to be a subshift over the alphabet $\{A,B\}$. We study the realization of binary subshifts as winning shifts, and the connections between properties of $X$ and its winning shift. In particular, connections between the entropies of $X$ and its winning shift are obtained. While our main interest lies in SFTs and sofic subshifts, we also briefly study the winning shifts of minimal subshifts, and obtain, for example, a characterization of Sturmian words in terms of the winning shifts of their orbit closures. We also define winning shifts for two-directional subshifts, and prove some individual results that hold in this formalism but not in the one-directional case, or vice versa. Our results are mainly of mathematical interest.

\section{Definitions}

\subsection{Standard Definitions}

Let $S$ be a finite set, called the \emph{alphabet}. We denote by $S^*$ the set of finite words over $S$, and by $\lambda$ the empty word of length $0$. For $n \in \N \cup \{ \infty \}$, denote $S^{\leq n} = \{ w \in S^* \;|\; |w| \leq n \}$. The set $S^\N$ of infinite state sequences, or \emph{configurations}, is called the \emph{one-directional full shift on $S$}. If $x \in S^\N$ and $i \in \N$, then we denote by $x_i$ the $i$th coordinate of $x$, and we adopt the shorthand notation $x_{[i,j]} = x_i x_{i+1} \ldots x_j$. If $w \in S^*$, we denote $w \sqsubset x$, and say that \emph{$w$ appears in $x$}, if $w = x_{[i,i+|w|-1]}$ for some $i$. For words $u,v \in S^*$, we use the notation $u v \INF = u v v v \cdots$. We make analogous definitions for two-directional configurations, that is, elements of the set $S^\Z$, and write $S^\I$ for either $S^\N$ or $S^\Z$, when both are applicable. For a letter $c \in S$ and $u \in S^*$, we denote by $|u|_c$ the number of occurrences of $c$ in $u$. We say that a language $L \subset S^*$ is \emph{left (right) extendable}, if for all $w \in L$ there exists $c \in S$ such that $cw \in L$ ($wc \in L$). We say $L$ is \emph{factor-closed} if whenever $w \in L$ and $v \sqsubset w$, we have $v \in L$.

We define a metric $d$ on $S^\I$ by setting $d(x,y) = 0$ if $x = y$, and by setting $d(x,y) = 2^{-i}$ where $i = \min \{ |j| \;|\; x_j \neq y_j \}$ if $x \neq y$. The topology defined by $d$ makes $S^\I$ a compact metric space. We define the \emph{shift map} $\sigma : S^\I \to S^\I$ by $\sigma(x)_i = x_{i+1}$. Clearly $\sigma$ is a continuous surjection (homeomorphism in the two-directional case) from $S^\I$ to itself.

A \emph{subshift} is a closed subset $X \subset S^\I$ with the property $\sigma(X) \subset X$. When it cannot be deduced from the context, we will explicitly state whether a subshift is one- or two-directional. We define $\B_k(X) = \{ w \in \Sigma^k \;|\; \exists x \in X : w \sqsubset x\}$ as the set of words of length $k$ appearing in $X$, and define the \emph{language} of $X$ as $\B(X) = \bigcup_{k \in \N} \B_k(X)$. Since a subshift is uniquely defined by its language, and every right (and left in the two-directional case) extendable and factor-closed language defines a subshift \cite{LiMa95}, we may write $X = \B^{-1}(L)$, if $\B(X)$ is the set of factors of the extendable language $L \subset S^*$. Alternatively, a subshift is defined by a set $F \in S^*$ of \emph{forbidden words} as the set of configurations $\X_F = \{ x \subset S^\I \;|\; \forall w \in F : w \not\sqsubset x \}$. If $F$ is finite, then $\X_F$ is \emph{of finite type} (SFT for short), if $F$ is a regular language, then $\X_F$ is \emph{sofic}, and if $\B(X)$ is (co-)recursively enumerable, then $X$ is \emph{recursively enumerable} (effective, respectively). A sofic shift which is not an SFT is called \emph{proper sofic}. Once a finite set of forbidden patterns is chosen, the length of the longest pattern is called the \emph{window size} of the corresponding SFT. For a two-directional subshift $X \subset S^\Z$, we denote $X_\N = \{ x_{[0, \infty)} \;|\; x \in X \} \subset S^\N$.

We say that a subshift $X$ is \emph{transitive} if for all $u,w \in \B(X)$ there exists $v \in \B(X)$ such that $uvw \in \B(X)$, and \emph{mixing} if the length of $v$ can be chosen arbitrarily, as long as it is sufficiently large (depending on $u$ and $w$). It is known that for mixing SFTs and sofic shifts, the length of $v$ can be chosen independently of $u$ and $w$, and the smallest such length is called the \emph{mixing distance of $X$}. A subshift $X$ is \emph{minimal} if it contains no proper nonempty subshift. The \emph{entropy} of a subshift $X$ is defined as $h(X) = \lim_{n \rightarrow \infty} \frac{1}{n} \log_2 |\B_n(X)|$.

An \emph{alternating finite automaton} is a quintuple $M = (Q,\Sigma,q_0,F,\delta)$, where $Q$ is the finite set of \emph{states}, $\Sigma$ the finite \emph{input alphabet}, $q_0 \in Q$ the \emph{initial state}, $F \subset Q$ the \emph{final states} and $\delta : (Q \times \Sigma) \to 2^{2^Q}$ the \emph{transition function}. For a word $w \in \Sigma^*$ and $q \in Q$, we say that $q(w)$ \emph{holds} if either
\begin{itemize}
	\item $w = \lambda$ and $q \in F$, or
	\item $w \neq \lambda$ and $\{ r \in Q \;|\; r(w_{[1,|w|-1]}) \mbox{ holds} \} \in \delta(q,w_0)$,
\end{itemize}
where $\lambda$ denotes the empty word. We say $w$ is \emph{accepted} by $M$ if $q_0(w)$ holds. Intuitively, the automaton processes the input word by removing its first letter, letting its every state recursively process the suffix, and then combining the results according to a Boolean function. \emph{Deterministic finite automata} can be seen as alternating automata where $\delta(q, c)$ is either empty or of the form $\{ P \subset Q \;|\; r \in P \}$ for all $q \in Q, c \in \Sigma$, or in other words, $q(w)$ holds if and only if $r(w_{[1, |w|-1]})$ does. In this case, we may see $\delta$ as having the type $\Delta \to Q$ for some $\Delta \subset Q \times \Sigma$, and write $\delta(q, c) = r$ for all $(q, c) \in \Delta$ as above. Note that this definition of DFAs is somewhat nonstandard, as the transition function may not be total. Alternating automata were first defined in \cite{ChKoSt81}, and there it was proved that they accept exactly the regular languages.

\subsection{Word Games}

We now define games in which two players take turns to pick letters from an alphabet and build a finite or infinite word. A \emph{word game} is a triple $(S,n,X)$, where $S$ is a finite set, $n \in \N \cup \{\N\}$ and $X \subset S^n$ is the \emph{target set}. If $X \subset S^*$, we may use $X$ in place of $X \cap S^n$, but this is always clear from the context. An \emph{ordered word game} is a tuple $(S,n,X,a)$, where $(S,n,X)$ is a word game and $a \in \{A,B\}^n$ is the \emph{turn order}. If $n = \N$ and $X$ is a subshift, we have \emph{subshift games} and \emph{ordered subshift games}, respectively. An ordered word game $(S,n,X,a)$ should be understood as the players $A$ and $B$ building a word in $S^n$ by choosing one coordinate at a time, with the coordinate $i$ being chosen by $a_i$. If the resulting word lies in $X$, then $A$ wins, and otherwise $B$ does.

Let $G = (S,n,X,a)$ be an ordered word game. A \emph{strategy for $G$} is a function $s : S^{\leq n-1} \to S$ that specifies the next pick of a player, given the word constructed thus far. The \emph{play} of a pair $(s_A,s_B)$ of strategies for $G$ is the sequence $x = p(G,s_A,s_B) \in S^n$ defined inductively by $x_i = s_{a_i}(x_{[0,i-1]})$. We say that a strategy $s$ is \emph{winning for $A$} if $p(G,s,s_B) \in X$ for all strategies $s_B$ ($A$ wins the game no matter how $B$ plays), and \emph{winning for $B$} if $p(G,s_A,s) \notin X$ for all strategies $s_A$. If $n \in \N$, or if $X$ is a closed set in the product topology of $S^\N$ (in particular if it is a subshift), then a winning strategy always exists for one of the players (the game is \emph{determined}) \cite{GaSt53}. As a side note, the game is determined even if $X$ is a general Borel set \cite{Ma75}.

\begin{example}
\label{ex:FiniteGame}
Let $S = \{0, 1\}$, $n = 3$ and $X = \{000, 110, 111\}$. Then $(S, n, X)$ is a word game, and $(S, n, X, BAB)$ is an ordered word game. In this game, $B$ has a winning strategy $s_B$ defined by
\[ s_B(\lambda) = 0, s_B(00) = s_B(01) = 1, \]
the other inputs being irrelevant. With this strategy for $B$, the ordered game starts with $B$ picking the letter $0$. Next, $A$ can pick either $0$ or $1$, and $B$ finishes with $1$. The possible outcomes are $001$ and $011$, neither of which lies in the target set $X$.
\end{example}

Given a set $X \subset S^n$, where $n \in \N \cup \{\N\}$, we define the \emph{winning set} of $X$ as
\[ W(X) = \{ a \in \{A,B\}^n \;|\; A \mbox{ has a winning strategy for } (S,n,X,a)\}. \]
The alphabet $S$ will always be clear from the context. For a language $X \subset S^*$, we denote $W(X) = \bigcup_{n \in \N} W(X \cap S^n)$. If $n = \N$ and $X$ is a subshift, $W(X)$ is called the \emph{winning shift} of $X$. We endow the alphabet $\{A, B\}$ with the order $A < B$, and it is clear that $W(X)$ is always downward closed with respect to the coordinatewise partial ordering defined as
\[ v \leq w \Leftrightarrow |v| = |w| \wedge \forall i \in [0, |v|-1] : v_i \leq w_i. \]

\begin{example}
\label{ex:FiniteWinningSet}
We continue with the word game of Example~\ref{ex:FiniteGame}. We already know that $BAB \notin W(X)$, and thus $BBB \notin W(X)$. Since $X$ is nonempty, we trivially have $AAA \in W(X)$. By a simple case analysis, we see that $W(X) = \{ AAA, AAB, BAA \}$.
\end{example}

In a sense, the notion of winning shifts generalizes one of the problem families studied in \cite{Pe11}: Given a periodic sequence $a \in \{A, B\}^\N$ and a natural parametrized class of subshifts $(X_i)_{i \in \mathcal{I}}$, for which parameters $i \in \mathcal{I}$ do we have $a \in W(X_i)$? To showcase our formalism, we state the two results of \cite{Pe11} which are of this form. Here, the alphabet of the games is $\{0, 1\}$.

\begin{theorem}[Theorem~1.4 of \cite{Pe11}]
For all $\epsilon > 0$ and $t \in \N$, we have
\[ (AB^t)^\infty \in W \left(\mathcal{X} \left( \Set{wuw \;|\; |w| > N_{\epsilon,t}, |u| \leq (2 - \epsilon)^{\frac{|w|}{t+1}}} \right) \right), \]
for large enough $N_{\epsilon,t} \in \N$.
\end{theorem}

\begin{theorem}[Theorem~1.5 of \cite{Pe11}]
For all $\epsilon > 0$ and $t \in \N$, we have
\[ (AB^t)^\infty \in W \left(\mathcal{X} \left(\Set{uv \;|\; |u| = |v| > N_{\epsilon,t}, H(u, v) < ((2t + 2)^{-1} - \epsilon)|u|} \right) \right), \]
for large enough $N_{\epsilon,t} \in \N$, where $H(u,v)$ denotes the Hamming distance of $u$ and $v$, that is, the number of coordinates in which they differ.
\end{theorem}

We note that in terms of symbolic dynamics, $(AB^t)^\infty$ is a rather trivial object, while the subshifts on the right are quite complicated in that they are not sofic. Our take in this article is, simply put, to consider a simpler class of subshifts on the right side of the equation, but try to understand the left side more completely.

\section{Qualitative Results}

We begin by considering winning sets of languages, and then establish a connection between the winning sets of a subshift and its language. For an arbitrary language $L \subset S^*$, there might not be any meaningful relations between the finite sublanguages $L_n = L \cap S^n$ for $n \in \N$, so the same holds for the winning sets $W(L_n)$. Conversely, if the $L_n$ have more structure, we are sometimes able to `transfer' it to $W(L)$.

\begin{lemma}
\label{lem:FactorClosed}
Let $L \subset S^*$ be a factor-closed language. Then $W(L)$ is factor-closed.
\end{lemma}

\begin{proof}
Let $\lambda \neq a \in W(L)$, denote $n = |a|$ and let $s : S^{\leq n-1} \to S$ be a winning strategy for $A$ in the ordered game $(S, n, L, a)$. Then, since $L$ is factor-closed, the restriction of $s$ to $S^{\leq n-2}$ is clearly a winning strategy for $A$ in the ordered game $(S, n-1, L, a_{[0, n-2]})$. On the other hand, the strategy $v \mapsto s(s(\lambda)v)$ is winning for $A$ in the ordered game $(S, n-1, L, a_{[1, n-1]})$: Here, $A$ plays on the word $v$ as if there was an extra symbol in front of it, chosen according to the strategy $s$ if $a_0 = A$. Thus both $a_{[0, n-2]}$ and $a_{[1, n-1]}$ are in $W(L)$, which implies that $W(L)$ is factor-closed.
\end{proof}

\begin{lemma}
\label{lem:RightExtendable}
Let $L \subset S^*$ be a right extendable language. Then $W(L)$ is right extendable.
\end{lemma}

\begin{proof}
Let $\lambda \neq a \in W(L)$, denote $n = |a|$ and let $s : S^{\leq n-1} \to S$ be a winning strategy for $A$ in the ordered game $(S, n, L, a)$. Then, the strategy
\[ v \mapsto \left\{
	\begin{array}{ll}
		s(v), & \mbox{if $|v| \leq n-1$} \\
		c \in S, & \mbox{if $|v| = n$ and $vc \in L$}
	\end{array}
\right. \]
is winning for $A$ in the game $(S, n+1, L, aA)$. Note that the symbol $c$ in the above definition exists, since $L$ is right extendable. Thus $aA \in W(L)$, and $W(L)$ is right extendable.
\end{proof}

\begin{example}
\label{ex:NotLeftExtendable}
There is a certain asymmetry in the definition of $W$ which manifests itself in the fact that the winning set of a left extendable language may not be left extendable. A concrete example is $L = 0^+ + 1^+$, whose winning set is easily seen to be $(B + A)A^*$.
\end{example}

As stated in the introduction, the winning shift $W(X)$ of a subshift $X$ is a subshift itself, and we now prove this fact.

\begin{proposition}
\label{prop:WDownwardClosed}
Let $X \subset S^\N$ be a subshift. Then $W(X)$ is a subshift, and we have $\B(W(X)) = W(\B(X))$.
\end{proposition}

\begin{proof}
Let $a \in \{A, B\}^\N$ be arbitrary. We claim that $a \in W(X)$ if and only if $a_{[0, n-1]} \in W(\B(X))$ for all $n \in \N$. Suppose first that $a_{[0, n-1]} \notin W(\B(X))$ for some $n \in \N$, so that $B$ has a winning strategy $s$ in the game $(S, n, \B(X), a_{[0, n-1]})$. But then any strategy $s'$ for the game $(S, \N, X, a)$ that satisfies $s'|_{S^{\leq n-1}} = s$ is winning for $B$, since a forbidden pattern is introduced after the first $n$ turns. Thus $a \notin W(X)$.

Suppose then that $a_{[0, n-1]} \in W(\B(X))$ for all $n \in \N$, and let $s_n$ be a winning strategy for $A$ in the game $(S, n, \B(X), a_{[0, n-1]})$. We use a standard compactness argument to construct a winning strategy $s$ for $(S, \N, X, a)$. First, there exists a symbol $c_0 \in S$ such that $s_n(\lambda) = c_0$ for infinitely many $n$; let $(s^0_n)_{n \in \N}$ be the subsequence of such strategies, and set $s(\lambda) = c_0$. Then, there exists a function $f_1 : S \to S$ such that for infinitely many $n$, we have $s^0_n(c) = f_1(c)$ for all $c \in S$; let $(s^1_n)_{n \in \N}$ be the subsequence of such strategies, and set $s|_S = f_1$. Inductively we define a function $f_i : S^i \to S$, a subsequence $(s^i_n)_{n \in \N}$ for which $s^i_n(v) = f(v)$ holds for all $v \in S^i$, and set $s|_{S^i} = f_i$, for all $i \in \N$.

We then claim that $s$ is winning for $A$ in the game $G = (S, \N, X, a)$. If not, there exists a strategy $s_B$ for $B$ such that the play $p(G, s, s_B) \notin X$. Since $X$ is a subshift, we actually have $p(G, s, s_B)_{[0, n-1]} \notin \B(X)$ for some $n \in \N$. By the definition of $s$, there exists $m \geq n$ such that $s_m$ agrees with $s$ on $S^{\leq n-1}$. But then the restriction of $s_B$ to $S^{\leq m-1}$ as a strategy of $B$ would beat $s_m$ in the game $(S, m, \B(X), a_{[0, m-1]})$, a contradiction with the definition of $s_m$. This proves the above claim.

Now it is clear from Lemma~\ref{lem:FactorClosed} and Lemma~\ref{lem:RightExtendable} that $W(X)$ is a subshift defined by the factor-closed, right extendable language $W(\B(X))$.
\end{proof}

Inspired by this result and the fact that winning sets of left extendable languages may not be left extendable, we make the following definition for two-directional subshifts.

\begin{definition}
Let $X \subset S^\Z$ be a two-directional subshift. We define $W(X) \subset \{A, B\}^\Z$ as the two-directional subshift $\X_F$, where $F = \{A, B\}^* - W(\B(X))$.
\end{definition}

It is easy to see that the winning shift $W(X)$ of a two-directional subshift $X \subset S^\Z$ is exactly the set of configurations $a \in \{A, B\}^\Z$ such that $a_{[i, \infty)} \in W(X_\N) \subset \{A, B\}^\N$ for all $i \in \Z$. Note also that while $W(X)$ is downward closed, it is usually not an ideal, that is, a downward closed subshift which is also closed under coordinatewise maxima of two configurations. Namely, the only binary two-directional subshifts which are ideals are $\{\INF A \INF\}$ and $\{A,B\}^\Z$. However, the set of downward closed binary subshifts is very rich in structure. We now present three examples of winning shifts exhibiting some interesting properties.

\begin{example}
\label{ex:WinningShifts}
In this example, all subshifts are two-directional unless otherwise noted.

Let $X = \B^{-1}(0^*1^*)$, a binary SFT defined by the single forbidden pattern $10$. Its winning shift is $W(X) = \B^{-1}(A^*BA^*)$: In a one-directional game $(S, \N, \B^{-1}(0^*1^*), a)$, it is clear that $B$ should always first play $1$ and then $0$, and that this strategy is winning if and only if $B$ occurs at least twice in $a$. Here, $W(X)$ is nontransitive and proper sofic.

Let then $Y = \B^{-1}(0^*(10^*20^*)^*)$, a mixing proper sofic shift. Its winning shift is also $W(Y) = \B^{-1}(A^*BA^*)$: If $B$ has only a single turn, $A$ can win by always playing the letter $0$. However, if $B$ gets to play twice, he wins by the following strategy: On his first turn, $B$ plays $1$, and on his second turn, he plays whichever of $1$ and $2$ was played last.

Finally, let $Z = \B^{-1}((01 + 0001)^*)$, a transitive SFT. We claim that the winning shift of $Z$ is $Z' = \B^{-1}((AB(AA)^+)^*)$. For that, let $a \in Z'$, and suppose without loss of generality that the $B$'s in $a$ occur at even coordinates. Then $A$ wins the game defined by $a_{[i,\infty)}$ by playing $0$ in odd and $1$ in even coordinates. Conversely, let $a \notin Z'$. Then either $B$ appears in $a$ in coordinates of both parity, in which case he wins by always playing $1$, or $a_i = a_{i+2} = B$ for some $i \in \Z$, so that $B$ wins by playing $00$ or $11$, depending on the situation. In this case, $W(Z)$ is transitive and proper sofic. We finally note that in the one-directional case, we would have $BABA \INF \in W(Z)$, even though $BABA$ was a forbidden pattern in the two-directional case. Compare this to Example~\ref{ex:NotLeftExtendable}.
\end{example}

Later, Corollary~\ref{cor:Soficity} and Proposition~\ref{prop:Mixing} will show that the winning shift of a transitive SFT must be a transitive sofic shift, so in that sense our third example is `maximally complicated.'

\begin{definition}
We denote by $\W$ the class of winning shifts of all subshifts (one- or two-directional, depending on the context).
\end{definition}

We now prove closure properties of some classes of subshifts under the operation $W$, and make some basic observations about $\W$. From now on, we will only rarely refer to strategies as concrete functions, and more often just informally state how a player behaves in certain situations.

\begin{proposition}
\label{prop:Regularity}
Let $X \subset S^*$ be a regular language. Then $W(X)$ is regular.
\end{proposition}

\begin{proof}
Since $X$ is regular, there exists a deterministic finite automaton $M = (Q, S, q_0, F, \delta)$ accepting it, where $\delta$ is defined on all of $Q \times S$. For $q \in Q$, denote by $L^M_q$ the language of the automaton $(Q, S, q, F, \delta)$, so that $X = L^M_{q_0}$. Define the alternating finite automaton $M' = (Q, \{A, B\}, q_0, F, \delta')$ by the transition function 
\begin{itemize}
	\item $\delta'(q, A) = \{ P \subset Q \;|\; \exists c \in S : \delta(q, c) \in P \}$
	\item $\delta'(q, B) = \{ P \subset Q \;|\; \forall c \in S : \delta(q, c) \in P \}$
\end{itemize}

We claim that the language accepted by $M'$ is $W(X)$, and prove this by showing that $W(L^M_q) = L^{M'}_q$ holds for all $q \in Q$. For this, let $a \in \{A, B\}^*$. If $a = \lambda$, then clearly $a \in W(L^M_q)$ iff $a \in L^{M'}_q$ by the definition of $M'$. Suppose then that $a = Ab$, where $b \in \{A, B\}^*$. Using the induction hypothesis on $b$, we get
\[ a \in W(L^M_q) \Leftrightarrow \exists c \in S : b \in W(L^M_{\delta(q, c)}) \Leftrightarrow \exists c \in S : b \in L^{M'}_{\delta(q, c)} \Leftrightarrow a \in L^{M'}_q. \]
The same argument applies for the case $a = Bb$, with the existential quantifiers replaced by universal ones. We have now proved that $W(X)$ is recognized by an alternating automaton, and thus is regular.
\end{proof}

In practice, the language $W(X)$ may be difficult to describe succinctly. The automaton $M'$ in the above proof has $|Q|$ states, and there then exists a deterministic automaton accepting $W(X)$ with at most $2^{2^{|Q|}}$ states. Furthermore, the reverse language of $W(X)$, or the set $W(X)^R = \{ a_{|a|-1} \cdots a_1 a_0 \;|\; a \in W(X) \}$, can be recognized by a deterministic automaton with just $2^{|Q|}$ states \cite{ChKoSt81}. These bounds are optimal for general alternating automata, but we have not checked whether this is the case for our construction. We also note that the reversal operation and $W$ do not usually commute or respect each other, as shown by the following example.

\begin{example}
\label{ex:Reversals}
Define the finite language $X = \{ 0000, 0011 \}$. An easy computation shows that $W(X) = \{ AAAA, AABA \}$ and $W(X^R) = \{ AAAA, BAAA \}$, and thus the sets $W(X)$, $W(X^R)$, $W(X)^R$ and $W(X^R)^R$ are all distinct.
\end{example}

Proposition~\ref{prop:Regularity} has the following immediate corollary.

\begin{corollary}
\label{cor:Soficity}
Let $X \subset S^\I$ be a sofic shift. Then $W(X)$ is sofic.
\end{corollary}

\begin{example}
\label{ex:EvenShift}
We show an example of computing the winning shift of a sofic shift using finite automata. Let $X = \B^{-1}(((00)^*1)^*)$ be the two-directional \emph{even shift}. We compute its winning shift $W(X)$. In Figure~\ref{fig:EvenDFA}, a DFA accepting $\B(X)$ is depicted. In Figure~\ref{fig:EvenWin}, we have followed the construction of Proposition~\ref{prop:Regularity} to obtain an alternating automaton for $W(\B(X))$, and in Figure~\ref{fig:EvenWinDFA}, we have used the standard subset construction of \cite{ChKoSt81} to obtain a DFA for $W(\B(X))^R$. From this, we finally conclude that $W(X) = \B^{-1}(A^*BB(A + AB)^*)$, since this language is left extendable. Interestingly, this subshift is almost equal to the well-known \emph{golden mean shift} $\B^{-1}((A + AB)^*)$, which is weakly conjugate to the even shift, see \cite[Section 1.5]{LiMa95} for more information.
\end{example}

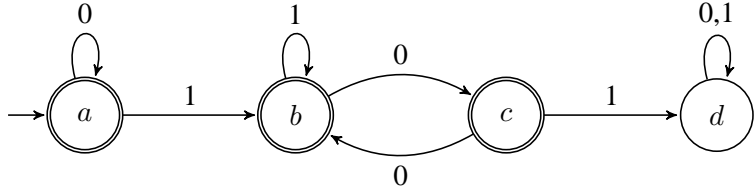
\begin{figure}
\centering
\begin{tikzpicture}[->,>=stealth',shorten >=1pt,auto,node distance=2.8cm,
                    semithick]
  \tikzstyle{every state}=[draw]

  \node[state, initial, initial text=, accepting] (a) {$a$};
  \node[state, accepting] (b) [right of=a] {$b$};
  \node[state, accepting] (c) [right of=b] {$c$};
  \node[state] (d) [right of=c] {$d$};
  
  \path (a) edge [loop above] node {0} (a)
			  (a) edge [] node {1} (b)
  			(b) edge [bend left] node {0} (c)
        (c) edge [bend left] node {0} (b)
        (b) edge [loop above] node {1} (b)
        (c) edge [] node {1} (d)
        (d) edge [loop above] node {0,1} (d);
\end{tikzpicture}
\caption{A DFA accepting the language of the even shift.}
\label{fig:EvenDFA}
\end{figure}

\begin{figure}
\centering
\begin{tikzpicture}[->,>=stealth',shorten >=1pt,auto,node distance=2.8cm,
                    semithick]
  \tikzstyle{every state}=[draw]

  \node[state, initial, initial text=, accepting] (a) {$a$};
  \node[state, accepting] (b) [right of=a] {$b$};
  \node[state, accepting] (c) [right of=b] {$c$};
  \node[state] (d) [right of=c] {$d$};
  
  \path (a) edge [loop above] node {$A,B$} (a)
			  (a) edge [] node {$A,B$} (b)
  			(b) edge [bend left] node {$A,B$} (c)
        (c) edge [bend left] node {$A,B$} (b)
        (b) edge [loop above] node {$A,B$} (b)
        (c) edge [] node {$A,B$} (d)
        (d) edge [loop above] node {$A,B$} (d);
\end{tikzpicture}
\caption{An alternating automaton for $W(\B(X))$ in Example~\ref{ex:EvenShift}. The $A$-transitions are existential, and the $B$-transitions are universal.}
\label{fig:EvenWin}
\end{figure}
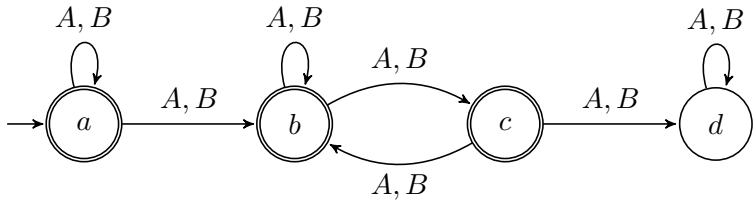

\begin{figure}
\centering
\begin{tikzpicture}[->,>=stealth',shorten >=1pt,auto,node distance=2.8cm,
                    semithick]
  \tikzstyle{every state}=[draw]

  \node[state, initial, initial text=, accepting] (abc) {$\{a,b,c\}$};
  \node[state, accepting] (ab) [right of=abc] {$\{a,b\}$};
  \node[state, accepting] (a) [right of=ab] {$\{a\}$};
  \node[state] (n) [right of=a] {$\emptyset$};
  
  \path (abc) edge [loop above] node {$A$} (abc)
			  (abc) edge [bend left] node {$B$} (ab)
			  (ab) edge [bend left] node {$A$} (abc)
  			(ab) edge [] node {$B$} (a)
  			(a) edge [loop above] node {$A$} (a)
  			(a) edge [] node {$B$} (n)
        (n) edge [loop above] node {$A,B$} (n);
\end{tikzpicture}
\caption{A DFA for $W(\B(X))^R$ in Example~\ref{ex:EvenShift}.}
\label{fig:EvenWinDFA}
\end{figure}
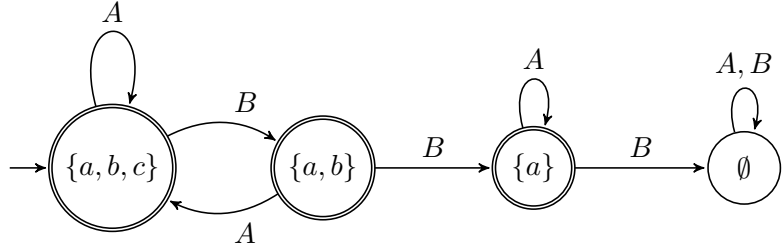

In Example~\ref{ex:WinningShifts} we presented an SFT whose winning shift was proper sofic. This raises the following question.

\begin{question}
What is the class of $W(X)$ for $X$ SFT?
\end{question}

In particular, it is not known whether all downward closed sofic shifts over $\{A, B\}$ can be realized as winning shifts of SFTs.

The following result is also easy to see.

\begin{proposition}
If a subshift $X \subset S^\N$ is recursively enumerable or effective, then so is $W(X)$.
\end{proposition}

\begin{proof}
We first note that the winning operator $W$ is clearly monotonous with respect to the inclusion order, that is, if $L \subset L' \subset S^*$, then $W(L) \subset W(L')$.

Let then $n \in \N$ and $a \in \{A, B\}^n$ be arbitrary. If $X$ is recursively enumerable, we can enumerate $\B_n(X)$, and since $W$ is monotonous, we can also recognize when $a \in W(\B_n(X))$. Similarly, if $X$ is effective, we can enumerate $S^n - \B_n(X)$ and recognize the case $a \notin W(\B_n(X))$.
\end{proof}

We now give a rather trivial full characterization of the class $\W$.

\begin{proposition}
If a subshift $X \subset \{A, B\}^\I$ is downward closed, then $W(X) = X$.
\end{proposition}

\begin{proof}
In a downward closed subshift, each player $P$ should always play the symbol $P$. 
\end{proof}

\begin{corollary}
\label{cor:WCharacterization}
The class of downward closed subshifts over $\{A, B\}$ is exactly $\W$. The class $\W$ is closed under union and intersection.
\end{corollary}

By Corollary~\ref{cor:WCharacterization}, the class of downward closed sofic (recursively enumerable, effective) subshifts over $\{A, B\}$ is exactly the class of winning shifts of sofic (recursively enumerable, effective) subshifts. Thus the classes of winning shifts of sofic shifts and recursive subshifts are also closed under union and intersection. However, the corollary is not very useful when considering classes of subshifts which are not closed under the operation $W$, in particular the SFTs. Therefore, we mention the following standard trick.

\begin{proposition}
Let $X \subset S^\I$ and $Y \subset R^\I$ be subshifts. Then $W(X \times Y) = W(X) \cap W(Y)$.
\end{proposition}

\begin{corollary}
The class of winning shifts of SFTs is closed under intersection.
\end{corollary}

We now return to the question of how different properties of $X$ affect the winning shift $W(X)$.

\begin{proposition}
\label{prop:MixingSFT}
If $X \subset S^\Z$ is a mixing (transitive) SFT, then $W(X)$ is mixing (transitive).
\end{proposition}

\begin{proof}
Let first $X$ be mixing with mixing distance and window size $r$. Let $a, b \in \B(W(X))$ and $k \geq 2r$. In her winning strategy for $A^rbA \INF$, $A$ starts by playing a word $w \in \B_r(X)$. We claim that $\INF AaA^kbA \INF \in W(X)$, and that the following is a winning strategy for $A$ on $A^iaA^kbA \INF$. In the first $i + |a|$ positions, $A$ plays as on $A^iaA \INF$, resulting in a word $u \in \B_{i+|a|}(X)$. In the next $k$ coordinates, she plays a word $v$ such that $uvw \in \B(X)$, and then continues as on $A^rbA \INF$. The strategy is winning for $A$, since the symbols played in the first $i + |a| + k - r$  coordinates have no effect on the game played on the tail $bA \INF$. This implies that $aA^kb \in \B(W(X))$, and thus $W(X)$ is mixing.

The case of transitive $X$ is slightly more complicated, but essentially similar. Here we need the following facts from \cite[Section 4.5]{LiMa95}, where $r$ is again the window size of $X$. There exists $p \in \N$, called the \emph{period of $X$}, and each word $w \in \B_r(X)$ has a \emph{period class} $c(w) \in [0, p-1]$ such that if $v \in \B_{r+1}(X)$, then $c(v_{[1,r]}) = c(v_{[0,r-1]}) + 1 \bmod p$. There exists $m \geq r$ such that for all $u, v \in \B_r(X)$ with the same period class, there exists $w \in \B_m(X)$ such that $uwv \in \B(X)$.

If now $a, b \in \B(W(X))$, we claim that $A$ has a winning strategy on the configuration $A^iaA^kbA \INF$ for some $k \in [m, m+p-1]$. Again until the final letter of $a$, $A$ should play as on $A^iaA \INF$, resulting in a word $u \in \B_{i + |a|}(X)$. Let then $w \in \B_r(X)$ be the word played by $A$ in her winning strategy for $A^rbA \INF$, and denote $d = c(w) - c(u_{[|u|-r-1, |u|-1]}) \bmod p$. If $i \geq r$, then $d$ does not depend on the strategy of $B$. We now let $k = m + d$, and let $A$ continue as on $A^iaA \INF$ in the next $d$ coordinates, resulting in a word $u' \in \B_{i + |a| + d}$ such that $c(w) = c(u'_{[|u'|-r-1, |u'|-1]})$. Next, $A$ plays a word $v \in \B_m(X)$ such that $u'vw \in \B(X)$, given by the previous paragraph, and proceeds as on $A^rbA \INF$.
\end{proof}

In Example~\ref{ex:WinningShifts}, we computed the winning shifts of some two-directional subshifts, and saw that Proposition~\ref{prop:MixingSFT} does not hold for general sofic shifts, and that the winning shift of a transitive SFT can be proper sofic. Note that the proposition does not hold as such for one-directional subshifts. We also have the following simple result, which can be used in the computation of some winning shifts.

\begin{proposition}
\label{prop:Mixing}
Suppose $X \subset S^\I$ is a mixing sofic shift with mixing distance $m$, and that $B \sqsubset W(X)$. Then $\B^{-1}((A^*A^{m+k}B)^*) \subset W(X)$ for some $k \in \N$. If $X$ is an SFT with window size $r$, we can choose $k = r-1$.
\end{proposition}

\begin{proof}
We prove the claim in the one-directional case, since the arguments used are the same in both. Since $B \sqsubset W(X)$, there exists $v \in \B(X)$ such that $vc \in \B(X)$ for all $c \in S$, and by \cite[Section 3.3]{LiMa95}, we can choose $v$ to be \emph{intrinsically synchronizing}, that is, $uv, vw \in \B(X)$ implies $uvw \in \B(X)$. Also, if $X$ is an SFT with window size $r$, we can choose $v \in \B_{r-1}(X)$.

Let $a \in \B^{-1}((A^*A^{m+|v|}B)^*) \subset \{A, B\}^\N$ be arbitrary. We claim that $A$ has a winning strategy for the game $(S, \N, X, a)$. Without loss of generality we may assume that $a_0 = B$. Now, suppose that after some turn of $B$, the word $w$ constructed so far is in $\B(X)$. Now $A$ gets to play in the next $n \geq m + |v|$ coordinates. Since $m$ is a mixing distance for $X$, she can play the word $uv$, where $u \in \B_{n - |v|}(X)$ is such that $wuv \in \B(X)$. By the assumption on $v$, we have $wuvc \in \B(X)$ for all $c \in S$, and thus $B$ cannot introduce a forbidden word in his turn. This proves the claim.
\end{proof}

Using the ideas of the above proposition, we recompute the winning shift of the even shift of Example~\ref{ex:EvenShift}. This approach for computing $W(X)$ from a sofic shift $X$ is more `hands-on' and uses less general arguments, but may sometimes be less tedious, especially if the deterministic automaton for $\B(X)$ has many states.

\begin{example}
Let $X$ be again the even shift, and suppose that $w \in \B(X)$. Depending on where the rightmost $1$ occurs in $w$ (and if it exists), either $v = w0$ or $v = w1$ satisfies $v0, v1 \in \B(X)$. As in the proof of Proposition~\ref{prop:Mixing}, we then have $\B^{-1}((A^+B)^*) \subset W(X)$. It is easy to see that if $a \in (A^+B)^*$ and $n \in \N$, then $A^nBBa \in W(\B(X))$: Whatever $B$ plays in the first two coordinates, $A$ can still recover. On the other hand, $BA^nBB \notin W(\B(X))$ for all $n \in \N$, since $B$ can win by playing $1$ in his first turn, and $01$ or $10$ on the last two turns, depending on the parity of the last $1$ played. This shows again that $W(X) = \B^{-1}(A^*BB(A + AB)^*)$.
\end{example}

The following example shows yet another difference between the one- and two-directional cases.

\begin{example}
Fix the alphabet $S$. It is easy to check that the set of subshifts of $S^\I$ is closed with respect to the Hausdorff metric, and thus compact. For subshifts $X, Y \subset S^\I$, the Hausdorff distance $d_H(X, Y)$ is at most $2^{-n}$ iff we have $\B_n(X) = \B_n(Y)$. From Proposition~\ref{prop:WDownwardClosed} we then immediately see that the winning operator $W$ is continuous with respect to $d_H$ in the one-directional case. However, in the two-directional case, $W$ is not continuous. To see this, let $X_p = \{ x \in \{0, 1\}^\Z \;|\; \sigma^p(x) = x \}$ be the subshift of all $p$-periodic binary configurations, and $X = \{0, 1\}^\Z$. Then the sequence $(X_p)_{p \in \N}$ clearly converges to $X$ in the Hausdorff metric. But $W(X) = \{A, B\}^\Z$ and $W(X_p) = \{\INF A \INF\}$ for all $p \in \N$, so $(W(X_p))_{p \in \N}$ certainly does not converge to $W(X)$.
\end{example}

\section{Small Winning Shifts}
\label{sec:Minimals}

In this section, we briefly study subshifts with very small winning shifts and discuss winning shifts of minimal subshifts. Note that the only minimal downward closed subshift is $\B^{-1}(A^*)$, so that the following characterizes subshifts with minimal winning shifts.

\begin{proposition}
\label{prop:Periodic}
Let $X \subset \{0,1\}^\I$ be a binary subshift. If $\I = \Z$, then $W(X) = \{\INF A \INF\}$ if and only if $X$ is periodic. If $\I = \N$, then $W(X) = \{A \INF\}$ if and only if $|\B_1(X)| = 1$
\end{proposition}

\begin{proof}
Simply note that in the two-directional case, periodicity of $X$ is equivalent to the condition that $x_{(-\infty, 0]}$ uniquely determines $x_1$ for all $x \in X$. This in turn is clearly equivalent to $W(X) = \{\INF A \INF\}$.

The second claim is obvious: if $BA \INF \notin W(X)$, then either $0$ or $1$ occurs in no point of $X$.
\end{proof}

Also the case of a winning shift with at most one $B$ in each point is interesting.

\begin{proposition}
\label{prop:Sturmian}
Let $X \subset \{0,1\}^\N$ be a binary subshift. Then $W(X) = \B^{-1}(A^*BA^*) \subset \{A, B\}^\N$ if and only if $|\B_n(X)| = n+1$ for all $n \in \N$.
\end{proposition}

\begin{proof}
For the second claim, we note that the condition that $|\B_n(X)| = n+1$ for all $n \in \N$ is equivalent to the following: for all $n \in \N$, there is exactly one $w \in \B_n(X)$ with $w0, w1 \in \B_{n+1}(X)$. Now, if we have $BA^kB \in \B(W(X))$ for some $k$, then there exist $u,v \in \B_k(X)$ such that $0ub, 1vb \in \B_{k+2}(X)$ for all $b \in \{0,1\}$, and the condition for the $|\B_n(X)|$ does not hold. Conversely, if there exist $w,w' \in \B_k(X)$ with $wb, w'b \in \B_{k+1}(X)$ and $w \neq w'$, then we can factorize the words as $w = u0v$ and $w' = u'1v$, so that $bvb' \in \B(X)$ for all $b, b' \in \{0,1\}$. This implies $BA^{|v|}B \in W(\B(X)) = \B(W(X))$.
\end{proof}

A simpler proof of the above result will be given in Section~\ref{sec:Entropy}. Subshifts $X \subset \{0,1\}^\N$ such that $|\B_n(X)| = n+1$ for all $n \in \N$ which are not eventually periodic are called \emph{Sturmian subshifts}. More often, \emph{Sturmian words} (words $x \in \{0, 1\}^\N$ such that $|\B_n(\overline{\mathcal{O}(x)})| = n + 1$ for all $n$) are studied instead of Sturmian subshifts. The subshift $\overline{\mathcal{O}(x)}$ is then automatically aperiodic, and thus a Sturmian subshift. Sturmian words have multiple characterizations and a rich theory, for which the interested reader can consult \cite{Lo02}. We rephrase Proposition~\ref{prop:Sturmian} in terms of infinite words.

\begin{corollary}
An infinite word $x \in \{0, 1\}^\N$ is Sturmian if and only if
\[ W(\overline{\mathcal{O}(x)}) = \B^{-1}(A^*BA^*) \subset \{A, B\}^\N. \]
\end{corollary}

We note that the characterization of the condition $W(X) = \B^{-1}(A^*BA^*)$ fails in the two-directional case, as shown by the following result.

\begin{definition}
Let $\tau : S \to S^*$ be a function, and extend it to $S^*$ by $\tau(uv) = \tau(u) \tau(v)$. We call $\tau : S^* \to S^*$ a \emph{substitution}, and if $\tau(s)$ has the same length for every $s \in S$, we say $\tau$ is \emph{uniform}. If there exists $n \in \N$ such that $s'$ occurs in $\tau^n(s)$ for every $s, s' \in S$, we say $\tau$ is \emph{primitive}. The \emph{subshift $X_\tau$ of $\tau$} has the language $\{ w \in S^* \;|\; \exists n \in \N, s \in S : w \sqsubset \tau^n(s) \}$, and we say $\tau$ is \emph{periodic} if $X_\tau$ is.
\end{definition}

\begin{proposition}
The two-directional subshift of a primitive uniform aperiodic substitution $\tau$ on $\{0,1\}$ has the winning shift $\B^{-1}(A^*BA^*)$.
\end{proposition}

\begin{proof}
Since $\tau$ is aperiodic, we have $\tau(0) \neq \tau(1)$, and we may assume without loss of generality that $\tau(0) = 0u$ and $\tau(1) = 1v$ for some $u,v \in S^{\ell-1}$ (if $\tau(0) = u0v$ and $\tau(1) = u1w$, we replace them with $0vu$ and $1wu$, respectively, and obtain the same subshift). From the results of \cite{Mo92}, we know that $\tau$ is \emph{unilaterally recognizable}, that is, there is $k \in \N$ such that for $x \in X_\tau$, the word $x_{[-k,-1]}$ uniquely determines whether or not $x_{[0,\infty)} = \tau(y_{[0,\infty)})$ for some $y \in X_\tau$.

Suppose now that $BA^{i-1}B \in \B(W(X_\tau))$ with $i \in \N$ minimal, and consider the game defined by $A^{n\ell}BA^{i-1}BA^{\ell-1}$, where $n\ell \geq k$. Here, $A$ plays a word $t \in \B_{n\ell}(X_\tau)$, after which $B$ can make a move, implying that $t = \tau(t')$ for some $t' \in \B(X_\tau)$ by the recognizability of $\tau$. Then $A$ continues, constructing either $t0r_0$ or $t1r_1$, then $B$ continues with $0$ or $1$. Again $t0r_0 = \tau(t'r_0')$ and $t1r_1 = \tau(t'r_1')$ for some $r_0', r_1' \in \B(X_\tau)$, and in particular $i \geq \ell$. Finally, $A$ finishes with either $u$ or $v$, resulting in one of
\[ t0r_00u, \quad t0r_01v, \quad t1r_10u, \quad t1r_11v. \]
Desubstituting by $\tau$, we have now shown that $t'r_0'b, t'r_1'b \in \B(X_\tau)$ for all $b \in \{0,1\}$, where $r_0'$ begins with $0$ and $r_1'$ with $1$. Since $|t'| = n$ can be arbitrarily large, and $|r'_b| < |br_b|$, we see that $BA^{j-1}B \in \B(W(X_\tau))$ for some $j < i$, a contradiction with the minimality of $i$.
\end{proof}

\begin{definition}
Let $K > 0$. A subshift $X \subset S^\I$ is \emph{linearly recurrent with constant $K$}, if for all $w \in \B(X)$, $x \in X$ and $i \in \I$, there exists $k < K|w|$ with $x_{[i+k, i+k+|w|-1]} = w$.
\end{definition}

Linearly recurrent subshifts were introduced in \cite{DuHoSk99}, and they have strong connections to the theory of substitutions. In particular, $X_\tau$ is known to be linearly recurrent for every primitive substitution $\tau$. We now show that the winning shift of a two-directional linearly recurrent subshift is also very small.

\begin{proposition}
\label{prop:LinearRecurrence}
Let $K > 0$, and let the subshift $X \subset S^\Z$ be linearly recurrent with constant $K$. Then for every $a \in \B(W(X))$ we have $|a|_B \leq \log_{|S|} (K(K+1))$.
\end{proposition}

\begin{proof}
We assume without loss of generality that $X$ is aperiodic, since otherwise $W(X) = \{\INF A \INF\}$ by Proposition~\ref{prop:Periodic}. From \cite[Theorem 24]{DuHoSk99} we know that $X$ is $K+1$-power free, that is, $w^{K+1} \notin \B(X)$ for all $w \in S^*$. Fix $x \in X$. Suppose that $|a|_B > \log_{|S|} (K(K+1)(1+n^{-1}))$ for some $a \in \B(W(X))$ and $n \in \N$, and consider the game defined by $A^{n|a|}a$. The possible outcomes of the game are of the form $uv$, where $u \in \B_{n|a|}(X)$, and there are at least $K(K+1)(1+n^{-1})+1$ choices for $v$. For each such $v$, there exists $k \in [0,K(n+1)|a|-1]$ such that $x_{[k,k+(n+1)|a|-1]} = uv$, and in particular $x_{[k,k+n|a|-1]} = u$. These $k$ are also pairwise distinct, and by the pigeonhole principle, some pair $k, k'$ of them satisfy
\[ d = |k-k'| \leq \frac{K(n+1)|a|}{K(K+1)(1+n^{-1})} = \frac{n|a|}{K+1}. \]
But then $u_i = u_{i+d}$ for all $i$, and thus $u = w^{K+1}u'$ for some $w \in \B_d(X)$ and $u' \in \B(X)$, a contradiction.
\end{proof}

We show by example that Proposition~\ref{prop:LinearRecurrence} is optimal in the sense that a global bound for the number of $B$'s cannot be found.

\begin{example}
For all $k, N \in \N$ there exists a linearly recurrent subshift $X \subset \{0, \ldots, k-1\}^\Z$ such that $|a|_B \geq N$ for some $a \in \B(W(X))$.
\end{example}

\begin{proof}
Define $R = \{0, \ldots, k^N-1\}$, and let $w = 000102 \cdots 0(k^N-1) \in R^*$. Define the primitive uniform substitution $\tau$ on $R$ by $\tau(i) = iw0$ for all $i \in R$. Then $\tau^n(0)i \in \B(X_\tau)$ for all $n \in \N$ and $i \in R$. Apply to $X_\tau$ the morphism $\rho$ that replaces each letter $i \in R$ by its base-$k$ representation, padded to length $N$. Now, in the game defined by $A^{N|\tau^n(0)|}B^N$, $A$ wins by playing the word $\rho(\tau^n(0))$ during her turns, since this word can be followed by $\rho(i)$ for any $i \in R$. This shows that $\INF A B^N A \INF \in W(\rho(X_\tau))$, and $\rho(X_\tau)$ is linearly recurrent, since $X_\tau$ is.
\end{proof}

Finally, we show that Proposition~\ref{prop:LinearRecurrence} cannot be generalized to all minimal systems.

\begin{example}
For any nontrivial alphabet $S$, there exists a minimal two-directional subshift $X \subset S^\Z$ such that $W(X)$ is uncountable.
\end{example}

\begin{proof}
We inductively define finite sets of words $W^i \subset S^*$ as follows: $W^1 = S$, and if $W^i = \{w^{(i)}_1, \ldots, w^{(i)}_{k_i}\}$ in increasing lexicographical order, then
\[ W^{i+1} = \{ w^{(i)}_1 \cdots w^{(i)}_{k_i} w^{(i)}_j w^{(i)}_{j'} \;|\; j, j' \in [1, k_i] \}. \]
Define the subshift $X \subset S^\Z$ as the set of points whose every factor occurs in some of the $w^{(i)}_j$. This subshift is minimal: The word $w^{(i)}_j$ occurs in the word $w^{(i+1)}_{j'}$ for any $j'$, and every point in $X$ is clearly some concatenation of such words by the inductive definition of the sets $W^i$.

It now suffices to produce an element of $W(X)$ with infinitely many $B$'s. For this, consider the points
\[ \cdots (w^{(4)}_1 \cdots w^{(4)}_{k_4} (w^{(3)}_1 \cdots w^{(3)}_{k_3} (w^{(2)}_1 \cdots w^{(2)}_{k_2} (w^{(1)}_1 \cdots w^{(1)}_{k_1} w^{(1)}_1 w^{(1)}_{b_1}) w^{(2)}_{b_2}) w^{(3)}_{b_3}) w^{(4)}_{b_4}) \cdots \]
These points are in $X$ since the words in parentheses are elements of the sets $W^i$ for all choices of $b_i \in [1,k_i]$. On the other hand, these points show that $B$ can play infinitely many times after the left tail
\[ \cdots w^{(5)}_1 \cdots w^{(5)}_{k_5} w^{(4)}_1 \cdots w^{(4)}_{k_4} w^{(3)}_1 \cdots w^{(3)}_{k_3} w^{(2)}_1 \cdots w^{(2)}_{k_2} w^{(1)}_1 \cdots w^{(1)}_{k_1},  \]
since there is a choice of an element of $W^1 = S$ in every $w^{(i)}_{b_i}$.
\end{proof}

\section{Entropy}
\label{sec:Entropy}

In this section, we study how the entropies of a subshift $X \subset S^\I$ and its winning shift relate to each other. These results hold for both one- and two-directional subshifts. We first observe that if $W(X)$ has high entropy, then in a typical point of $W(X)$, $B$ is able to play quite often, but $A$ still wins. Since $B$ can play arbitrarily, the $B$-coordinates can thus be chosen arbitrarily, with the resulting configuration still in $X$. This gives a lower bound for the entropy of $X$, which we approximate in the following.

\begin{lemma}
\label{lem:DensityToEntropy}
Let $X \subset S^\I$ be a subshift. Denoting
\[ d = \limsup_{n \rightarrow \infty} \left( \max \left\{ \frac{|a|_B}{n} \;\middle|\; a \in \B_n(W(X)) \right\} \right), \]
we have $h(X) \geq d \log_2 |S|$.
\end{lemma}

\begin{proof}
Let $\epsilon > 0$ and $n \in \N$ be arbitrary, and let $a \in \B(W(X))$ be such that $|a| \geq n$ and $\frac{|a|_B}{|a|} \geq d - \epsilon$. Let $s$ be a winning strategy for $A$ in the game $G = (S, |a|, \B_{|a|}(X), a)$. By enumerating all the possible strategies $s_B$ for $B$ and the prefixes of the resulting plays $p(G, s, s_B)$, we find that $|\B_{|a|}(X)| \geq |S|^{|a|_B}$. This implies that
\[ \frac{1}{|a|} \log_2 |\B_{|a|}(X)| \geq \frac{|a|_B}{|a|} \log_2 |S| \geq (d - \epsilon) \log_2 |S|. \]
Since $|a|$ can be chosen arbitrarily large and $\epsilon$ arbitrarily small, the claim follows. 
\end{proof}

In what follows, we will several times need the well-known approximation formula
\begin{equation}
\label{eq:BinomApprox}
\binom{n}{m} \leq \left(\frac{ne}{m}\right)^m
\end{equation}
for the binomial coefficient, which holds for all $m, n \in \N$.

\begin{lemma}
\label{lem:BinomApprox}
Let $1 < k < 2$, and let $1 > \epsilon > 0$ be such that $\left(\frac{e}{\epsilon}\right)^\epsilon < k$. Then $\binom{n}{\lfloor \epsilon n \rfloor} \leq k^n$ for all $n$ large enough.
\end{lemma}

\begin{proof}
We first use \eqref{eq:BinomApprox} to show that
\[ \binom{n}{\lfloor n \epsilon \rfloor}
\leq \left(\frac{ne}{\lfloor n \epsilon \rfloor}\right)^{\lfloor n \epsilon \rfloor}
\leq \left(\frac{ne}{\lfloor n \epsilon \rfloor}\right)^{n \epsilon}
= \left(\frac{n \epsilon}{\lfloor n \epsilon \rfloor}\right)^{n \epsilon} \left(\frac{e}{\epsilon}\right)^{n \epsilon}. \]
Then, we know that $\left(\frac{x}{\lfloor x \rfloor}\right)^x \leq \left(\frac{x}{x - 1}\right)^x = \left(1 - \frac{1}{x}\right)^{-x} \longrightarrow e$ as $x \longrightarrow \infty$, from which the claim follows for large enough $n$.
\end{proof}

\begin{proposition}
\label{prop:Entropy}
Let $X \subset S^\I$ be a subshift with $h(W(X)) \geq \log_2 k$, and let $1 > \epsilon > 0$ be such that $\left(\frac{2e}{\epsilon}\right)^\epsilon < k$. Then $h(X) \geq \epsilon \log_2 |S|$. In particular, if $h(W(X)) > 0$, then $h(X) > 0$.
\end{proposition}

\begin{proof}
First, note that $W(X)$ is binary, so $1 \leq k \leq 2$. If we have $k = 2$, then $W(X) = \{A, B\}^\Z$ and $X = S^\Z$, and the claim holds. Suppose then that $k < 2$, and let $n \in \N$ be so large that $\binom{n}{\lfloor \epsilon n \rfloor} \leq k^n$ (given by Lemma~\ref{lem:BinomApprox}). Let also $\left(\frac{2e}{\epsilon}\right)^\epsilon < \ell < k$. If we denote $d_n = \max \{ |a|_B \;|\; a \in \B_n(W(X)) \}$ and suppose $d_n \leq \epsilon n$, we get an upper bound for $|\B_n(W(X))|$ as follows. We choose $d_n$ coordinates from the interval $[0, n-1]$, and then choose a subset of these to contain the symbol $B$, setting every other coordinate to $A$. Using~\eqref{eq:BinomApprox}, we obtain the bound
\[ |\B_n(W(X))| \leq \binom{n}{d_n} 2^{d_n} \leq \left(\frac{2e}{\epsilon}\right)^{\epsilon n} < \ell^n, \]
since the function $\left( \frac{2e}{\epsilon} \right)^\epsilon$ is increasing on the open interval $(0,1)$ (its derivative, $\left( \frac{2e}{\epsilon} \right)^\epsilon (\log \frac{2}{\epsilon} - 1)$, is positive). Since $h(W(X)) \geq \log_2 k$, we must thus have $d_n > \epsilon n$ for infinitely many $n$, so that
\[ \limsup_{n \rightarrow \infty} \left( \max \left\{ \frac{|a|_B}{n} \;\middle|\; a \in \B_n(W(X)) \right\} \right) \geq \epsilon. \]
Then, Lemma~\ref{lem:DensityToEntropy} gives the claim.
\end{proof}

In Figure~\ref{fig:Graph} we have plotted the bound given by the above proposition.

In the binary case, the entropy of a subshift and its winning shift are in fact equal. This actually follows from the stronger result that in binary games of finite length, the target set and winning set are of equal size. We have already seen this phenomenon in Example~\ref{ex:FiniteWinningSet} and Example~\ref{ex:Reversals}.

\begin{proposition}
\label{prop:SameSize}
For all $n \in \N$ and $L \subset \{0,1\}^n$, we have $|L| = |W(L)|$.
\end{proposition}

\begin{proof}
We prove this by induction, starting with the case $n = 1$, which is easily seen true. Suppose then that $n > 1$, and let $L_c = \{ w \in \{0,1\}^{n-1} \;|\; cw \in L \}$ for $c \in \{0,1\}$. We clearly have $L = 0L_0 \cup 1L_1$, and by the induction hypothesis, also $|W(L_c)| = |L_c|$ holds for $c \in \{0,1\}$.

Let $a \in W(L)$, and suppose that $a_0 = A$. If $A$ has a winning strategy in which she starts with $c \in \{0,1\}$, then $a_{[1,n-1]} \in W(L_c)$. Conversely, if $a_{[1,n-1]} \in W(L_c)$, then $A$ has a winning strategy that starts with $c$, and then follows the strategy for $a_{[1,n-1]}$. Thus
\[ |\{ a \in W(L) \;|\; a_0 = A \}| = |W(L_0)| + |W(L_1)| - |W(L_0) \cap W(L_1)|. \]
On the other hand, if $a_0 = B$, then $a_{[1,n-1]}$ must be in $W(L_0) \cap W(L_1)$ and the converse also holds, so
\[ |\{ a \in W(L) \;|\; a_0 = B \}| = |W(L_0) \cap W(L_1)|. \]
All in all, we have that $|W(L)| = |W(L_0)| + |W(L_1)| = |L_0| + |L_1| = |L|$, and the claim is proved.
\end{proof}

\begin{corollary}
\label{cor:BinaryEntropy}
If $X \subset S^\I$ is a binary subshift, then $h(X) = h(W(X))$.
\end{corollary}

In a sense, the winning operator $W$ `rearranges' a binary language or subshift, so that it becomes downward closed. However, no natural bijection between even a finite constant-length language and its winning set seems to exist. Note that Proposition~\ref{prop:Sturmian} follows as a corollary: $\B^{-1}(A^*)$ is the only downward closed unary subshift, and $\B^{-1}(A^*BA^*)$ is the only downward closed subshift with exactly $n+1$ words of length $n$ for all $n \in \N$. This shows that there certainly is no natural bijection between a binary subshift and its winning shift, as Sturmian subshifts are uncountable but their winning shifts are not. Using Proposition~\ref{prop:SameSize}, we also obtain a slight simplification of the classical proof for the fact that if a subshift $X \subset \{0,1\}^\N$ satisfies $|\B_n(X)| \leq n$ for some $n \in \N$, then $X$ is eventually periodic: If this is the case, then for all $a \in W(X)$ and $i \geq n$ we have $a_i = A$. This implies that every $x \in X$ is determined by $x_{[0,n-1]}$, so $X$ is finite, hence eventually periodic.

Recall once again the even shift $X$ of Example~\ref{ex:EvenShift}, with winning shift $Y = \B^{-1}(A^*BB(A^+B)^*)$. We noted earlier that this winning shift is almost equal to the golden mean shift $Z = \B^{-1}((A^+B)^*)$, and it is easy to see that all three shifts have the same entropy, $\log_2 \frac{1 + \sqrt{5}}{2}$.

With a more general type of game, we show that every subshift $X \subset S^\I$ can be rearranged into a downward closed subshift with the same entropy. We will then use this construction to obtain deeper results about the connection of the entropies $h(X)$ and $h(W(X))$.

\begin{definition}
Let $n \in \N \cup \{\N\}$, and let $X \subset S^n$ (a subshift if $n = \N$), and let $a \in [1,|S|]^n$. We define another game in which $A$ and $B$ build a word $w \in S^n$. At each coordinate $i$, $A$ chooses a set $S_i \subset S$ with $|S_i| = a_i$, and $B$ chooses the symbol $w_i \in S_i$. We say $A$ wins the game if $w \in X$. We then denote $\tilde W(X) = \{ a \in [1,|S|]^n \;|\; \mbox{$A$ has a winning strategy on $a$} \}$, and call it the \emph{counting winning set} of $X$. As before, we define $\tilde W(L) = \bigcup_{n \in \N} \tilde W(S^n \cap L)$ for a language $L \subset S^*$, and prove that if $X \subset S^\N$ is a subshift, then so is $\tilde W(X)$, and $\B(\tilde W(X)) = \tilde W(\B(X))$. We then define counting winning sets for two-directional subshifts analogously to the winning sets.
\end{definition}

If $|S| = 2$, then $W(X) = \tilde W(X)$ up to renaming the symbols. The following proposition is a generalization of Corollary~\ref{cor:BinaryEntropy}.

\begin{proposition}
\label{prop:IlkkasGeneralization}
If $n \in \N$ and $L \subset S^n$, then $|\tilde W(L)| = |L|$ and $W(\tilde W(L)) = W(L)$.
\end{proposition}

\begin{proof}
For $n \leq 1$ this is trivial, so suppose $n > 1$. Let $L_c = \{ w \in S^{n-1} \;|\; cw \in L \}$. Then $L = \bigcup_{c \in S} L_c$, and $|L_c| = |\tilde W(L_c)|$ by the induction hypothesis.

Let $a \in \tilde W(L)$, and denote $k = a_0$. Then there exists a set $R \subset S$ with $|R| = k$ such that $a_{[1,n-1]} \in \tilde W(L_c)$ for all $c \in R$, so that $A$ can win on the suffixes. The converse also holds, and thus
\[ \left|\Set{ a \in \tilde W(L) \;|\; a_0 = k } \right| = \left|\Set{ a \in S^{n-1} \;|\; |\{ c \in S \;|\; a \in \tilde W(L_c) \}| \geq k } \right|. \]
This implies
\begin{align*}
|\tilde W(L)| &= \sum_{k=1}^{|S|} \left| \Set{ a \in S^{n-1} \;|\; |\{ c \in S \;|\; a \in \tilde W(L_c) \}| \geq k } \right| \\
&= \sum_{c \in S} |\tilde W(L_c)| = \sum_{c \in S} |L_c| = |L|
\end{align*}
by the induction hypothesis and some basic combinatorial identities.

Let $a \in \{A, B\}^n$ be arbitrary. Now $\tilde W(L)$ is downward closed, so in the game $([1,|S|], n, \tilde W(L), a)$, $A$ should always play the letter $1$ and $B$ should always play $|S|$. Thus $a \in W(\tilde W(L))$ if and only if $\tau(a) \in \tilde W(L)$, where $\tau$ is the letter-to-letter morphism $A \mapsto 1$ and $B \mapsto |S|$. But $\tau(a) \in \tilde W(L)$ is clearly equivalent to $a \in W(L)$ by the definition of the counting winning set.
\end{proof}

Again, we have the following corollary for subshifts.

\begin{corollary}
If $X \subset S^\I$ is a subshift, then $h(\tilde W(X)) = h(X)$ and $W(\tilde W(X)) = W(X)$.
\end{corollary}

We can now show that even in the general case $L \subset S^n$ we always have $|W(L)| \leq |L|$. This is because $W(L) = W(\tilde W(L))$ by Proposition~\ref{prop:IlkkasGeneralization}, and clearly $|W(\tilde W(L))| \leq |\tilde W(L)| = |L|$. As a corollary, $h(X) \geq h(W(X))$ holds for all subshifts $X \subset S^\I$.

Proposition~\ref{prop:Entropy} gave a lower bound for the entropy of $X \subset S^\I$, when $h(W(X))$ is known. We will next study the converse problem of bounding $h(W(X))$, when $h(X)$ and $S$ are fixed. The above corollary lets us assume that $X$ is downward closed.

Suppose that $X \subset S^\I$ is a subshift with $h(X)$ close to $\log_2 |S|$. Intuitively, we should then find configurations of $W(X)$ in which the letter $B$ appears often, giving $W(X)$ positive entropy too. Using Proposition~\ref{prop:IlkkasGeneralization}, we show that this is indeed the case.

\begin{proposition}
\label{prop:UpperBound}
Let $X \subset S^\I$ be a subshift with $h(X) > \log_2(|S|-1)$, and let $\epsilon > 0$ be such that $\left( \frac{2 e}{\epsilon} \right)^\epsilon < \frac{2^{h(X)}}{|S|-1}$. Then $h(W(X)) \geq \epsilon$.
\end{proposition}

\begin{proof}
By Proposition~\ref{prop:IlkkasGeneralization}, we can again assume that $S = \{0, \ldots, k\}$ and $X$ is downward closed. Then $W(X)$ is the image of $X$ under the letter-to-letter map defined by $k \mapsto B$ and $c \mapsto A$ for all $c \neq k$.

Let now $n \in \N$, and consider the size of $\B_n(X)$. Suppose that the maximal value of $|w|_k$ for $w \in \B_n(X)$ is $\delta n$. Then, we can give an upper bound for $|\B_n(X)|$ as follows. In a word $w \in S^n$, we first choose $\delta n$ coordinates arbitrarily, then choose some subset of these to contain the letter $k$, and finally choose the remaining (possibly all $n$) coordinates to have any value from $\{0, \ldots, k-1\}$. This way, we obtain at least all the words of $\B_n(X)$, so we have
\[ \B_n(X) \leq \binom{n}{\delta n} 2^{\delta n} k^n \leq \left( k \left(\frac{2e}{\delta}\right)^\delta \right)^n \]
using \eqref{eq:BinomApprox}, and if this holds for the same $\delta$ for infinitely many $n$, then $h(X) \leq \log_2 k \left(\frac{2e}{\delta}\right)^\delta$. Thus for infinitely many $n$, the above does \emph{not} hold for any $\delta \leq \epsilon$. But then the maximal value of $|w|_k$ for $w \in \B_n(X)$ is at least $\epsilon n$, so $\B_n(W(X)) \geq 2^{\epsilon n}$, which implies that $h(W(X)) \geq \epsilon$.
\end{proof}

We remark here that the condition $h(X) > \log_2(|S|-1)$ is necessary, since the SFT defined by the forbidden patterns $ss$ for all $s \in S$ has entropy $\log_2 (|S|-1)$, but its winning shift is the singleton $\{\INF A \INF\}$.

The bounds given by this proposition for some different values of $k$ are drawn in Figure~\ref{fig:Graph}, in the upper left corner. Note that the bound for $k = 1$, or $|S| = 2$, is redundant, as we already know by Corollary~\ref{cor:BinaryEntropy} that the binary case is restricted to the diagonal. It is nevertheless shown in the figure to remind the reader that our bounds are probably far from optimal. That said, we also conjecture that when the entropy of $X$ is very close to its theoretical maximum $\log_2 |S|$, the same should hold for $W(X)$.

\begin{conjecture}
Fix the alphabet $S$. For all $\epsilon > 0$ there exists $\delta > 0$ such that if $X \subset S^\I$ is a subshift with $h(X) > \log_2 |S| - \delta$, then $h(W(X)) > 1 - \epsilon$.
\end{conjecture}

\begin{figure}[ht]
\includegraphics[width=10cm]{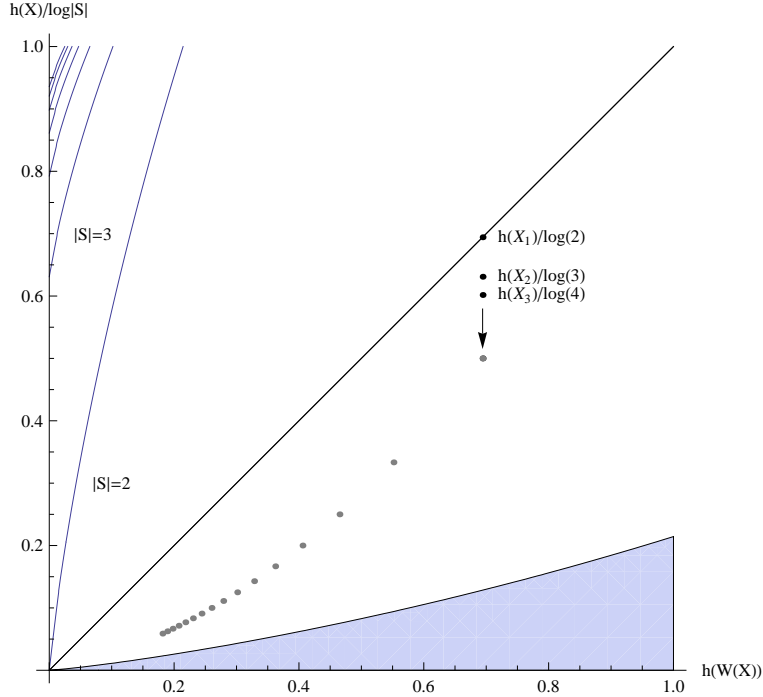}
\centering
\caption{Our results in a graphical form. The shaded area is where $h(X) / \log_2 |S|$ cannot lie by Proposition~\ref{prop:Entropy}, the diagonal line is given by Corollary~\ref{cor:BinaryEntropy} in the binary case, the lines in the upper left corner are the bounds given by Proposition~\ref{prop:UpperBound} and the points correspond to the numerical results of Section~\ref{sec:Numerical}.}
\label{fig:Graph}
\end{figure}

\section{Numerical Results}
\label{sec:Numerical}

In this section, we try to reach the lower bound given by Proposition~\ref{prop:Entropy} on $h(X) / \log_2 |S|$, when the entropy of $W(X)$ is known. This involves approximating and explicitly computing the entropies of certain SFTs, and the methods we use are related to those of \cite{Sh10}.

\begin{definition}
Let $X \subset \{0, 1\}^\I$ be a downward closed subshift, and let $S$ be a finite alphabet containing the letter $0$. Then we call the subshift
\[ \{y \in S^\I \;|\; \exists x \in X: \forall i \in \I: x_i = 0 \implies y_i = 0\} \]
the \emph{$S$-extension of $X$}, denoted by $E_S(X)$. We also denote $E_k(X) = E_{\{0, \ldots, k\}}(X)$.
\end{definition}

It is easy to see that $W(X) = W(E_S(X))$ for all binary subshifts $X$ and alphabets $S$.

\begin{proposition}
\label{prop:LowerES}
Let $X \subset S^\I$ be a subshift, where $0 \in S$. Then there exists a binary subshift $Y$ such that $h(E_S(Y)) \leq h(X)$ and $W(E_S(Y)) = W(X)$.
\end{proposition}

\begin{proof}
Applying Proposition~\ref{prop:IlkkasGeneralization}, we can assume that $S = \{0, \ldots, k\}$ and $X$ is downward closed. Consider the image $Y \subset \{0, 1\}^\I$ of $X$ under the symbol map defined by $k \mapsto 1$ and $c \mapsto 0$ for all $c \neq k$. Now, all of $X$, $Y$ and $E_S(Y)$ are downward closed, so $A$ should always play $0$ and $B$ should play $k$ (or $1$) in any game on $X$ or $E_S(Y)$ (or $Y$, respectively). It is then easy to see that $W(X) = W(Y) = W(E_S(Y))$.

As for the entropy, simply note that $E_S(Y) \subset X$, so $h(E_S(Y)) \leq h(X)$.
\end{proof}

This shows that it is enough to consider subshifts of the form $E_S(Y)$ when trying to reach a lower bound for the possible relative entropies of $X$, when the winning shift $W(X)$ (or its entropy) is fixed. We continue by computing the values of the entropies of $X_k = E_k(X)$ relative to $\log_2 (k+1)$ for all $k$, when $X$ is once again the golden mean shift. The language of $X_k$ is recognized by the DFA
\begin{center}
\begin{tikzpicture}[->,>=stealth',shorten >=1pt,auto,node distance=2.8cm,
                    semithick]
  \tikzstyle{every state}=[draw,inner sep=1pt,minimum size=20pt]

  \node[state,accepting] (a) {};
  \node[state,initial,initial text=,accepting] (b) [left of=a] {};
  
  \path (a) edge [bend left] node {$0$} (b)
			  (b) edge [loop above] node {$0$} (b)
			  (b) edge [bend left] node {$1, \ldots, k$} (a);
\end{tikzpicture}
\end{center}
It is known from from classical symbolic dynamics \cite{LiMa95} that the entropy of $X_k$ is exactly the logarithm of the Perron-Frobenius eigenvalue $\frac{1}{2}(1 + \sqrt{1 + 4k})$ of the adjacency matrix $\left( \begin{smallmatrix} 0 & 1 \\ k & 1 \end{smallmatrix} \right)$ of this automaton. Thus the relative entropy of $X_k$ is
\begin{equation}
\label{eq:hXk}
\frac{h(X_k)}{\log_2(k+1)} = \frac{\log_2 \left( \frac12 + \sqrt{\frac14 + k} \right)}{\log_2(k+1)} \stackrel{k \rightarrow \infty}{\longrightarrow} \frac12.
\end{equation}
We have computed some of the values of $h(X_k)/\log_2(k+1)$ in Figure~\ref{fig:Graph}. As we have $W(X_k) = W(X)$ for all $k \in \N$, the pairs $(h(W(X_k)), h(X_k))$ lie on the same vertical line in the figure.

We then generalize this result to the binary subshift $X_{(m)}$ with forbidden words $10^i1$ for all $i < m$, and show that the relative entropies $\frac{h(E_k(X_{(m)}))}{\log_2 (k+1)}$ approach the value $\frac{1}{m+1}$ as $k$ increases. The $X_{(m)}$ are examples of \emph{gap shifts}, a well-known class of binary subshifts (see \cite{LiMa95} for more information). The language of $E_k(X_{(m)})$ is recognized by the DFA
\begin{center}
\begin{tikzpicture}[->,>=stealth',shorten >=1pt,auto,node distance=2cm,
                    semithick]
  \tikzstyle{every state}=[inner sep=1pt,minimum size=20pt]

  \node[draw,state,accepting] (a0) {$0$};
  \node[draw,state,accepting] (a1) [right of=a0] {$1$};
  \node (a) [right of=a1] {$\cdots$};
  \node[draw,state,accepting] (am') [right of=a] {$m-1$};
  \node[draw,state,initial,initial text=,accepting] (am) [above of=a1] {$m$};
  
  \path (a0) edge node {$0$} (a1)
			  (a1) edge node {$0$} (a)
			  (a) edge node {$0$} (am')
			  (am') edge node [swap] {$0$} (am)
			  (am) edge [loop right] node {$0$} (am)
			  (am) edge node [swap] {$1, \ldots, k$} (a0);
\end{tikzpicture}
\end{center}
As above, we wish to compute the Perron-Frobenius of the adjacency matrix of this DFA, which is
\[
\left(
\begin{array}{ccccccc}
	0 & 1      & 0 &        & 0 & 0 \\
	0 & 0      & 1 & \cdots & 0 & 0 \\
	0 & 0      & 0 &        & 0 & 0 \\
	  & \vdots &   & \ddots & \multicolumn{2}{c}{\vdots} \\
	0 & 0      & 0 &  & 0 & 1 \\
	k & 0      & 0 & \multirow{-2}{*}{$\cdots$} & 0 & 1
\end{array}
\right)
\]
The eigenvalues of this matrix are exactly the roots of the polynomial $p_k(x) = x^{m+1} - x^m - k$, which we now analyze.

Let $\xi_k > 0$ be the largest positive root of $p_k(x)$ for a given $k \in \N$, when $m$ is fixed. By the classical Puiseux's Theorem, for large enough $k$ the root can be represented as a Puiseux series
\[ \xi_k = \sum_{i = \ell}^\infty c_i k^{-\frac{i}{q}} \]
for some $q \in \N$, $\ell \in \Z$ and coefficients $c_i \in \R$ that are independent of $k$. Now, since $\xi_k$ is a root of $p_k(x)$, the series satisfies the equation
\begin{equation}
\label{eq:XiSeries}
\left( \sum_{i = \ell}^\infty c_i k^{-\frac{i}{q}} \right)^{m+1} - \left( \sum_{i = \ell}^\infty c_i k^{-\frac{i}{q}} \right)^m = k.
\end{equation}
It is easy to see that $\ell < 0$, and then expanding the products gives $-\frac{\ell}{q} = \frac{1}{m+1}$ and $c_\ell = 1$ for the leading term. By a straightforward calculation, the second highest term of the left hand side of \eqref{eq:XiSeries} then equals
\[ (m+1) c_{\ell+1} k^{\frac{m}{q}} - k^{\frac{m}{m+1}}, \]
and since this term must vanish, we have $q = m+1$ and $c_{\ell+1} = 1/(m+1)$. By the above, this also implies that $\ell = -1$. Thus we have obtained the asymptotic formula
\[ \xi_k = \sqrt[m+1]{k} + \frac{1}{m+1} + o(1), \]
where the exact value of the $o(1)$-term is expressible as a power series in the variable $k^{-\frac{1}{m+1}}$ with no constant term. Since $h(E_k(X_{(m)})) = \log_2 \xi_k$, this in particular implies
\[ \frac{h(E_k(X_{(m)}))}{\log_2(k+1)} = \frac{\log_2 \left( \!\! \sqrt[m+1]{k} + O(1) \right)}{\log_2(k+1)} \stackrel{k \rightarrow \infty}{\longrightarrow} \frac{1}{m+1}, \]
which is exactly what we wanted.

We then estimate the entropy of the subshift $W(X_{(m)})$. This is equal to the entropy $h(X_{(m)})$, which is given by the logarithm of the largest root $\xi$ of $p_1(x) = x^{m+1} - x^m - 1$. It is easy to see that $1 < \xi < 2$ when $m \geq 2$, so we denote $\xi = 1 + \epsilon$, where $0 < \epsilon < 1$. We obtain the equation $\epsilon (1 + \epsilon)^m = 1$, and taking logarithms, this becomes $\log \epsilon + m \log (1 + \epsilon) = 0$. Using the Taylor series for $\log (1 + \epsilon)$, the left hand side is equal to
\[
\log \epsilon + m \epsilon + m t \epsilon^2,
\]
where $-\frac12 \leq t \leq \frac12$. It is easy to see that for large enough $m$, $\log x + m x + m t x^2$ is positive when $x = m^{-1}(\log m + 2 \log \log m)$, and negative when $x = m^{-1}(\log m - 2 \log \log m)$, and thus the root $\epsilon$ lies somewhere in between. Then $\epsilon = \frac{\log m}{m} + O\left( \frac{\log \log m}{m} \right)$, implying that
\[ h(W(X_{(m)})) = h(X_{(m)}) = \log_2 \left( 1 + \frac{\log m}{m} + O\left( \frac{\log \log m}{m} \right) \right). \]

We have approximated the values of $h(W(X_{(m)})) = h(X_{(m)})$ for $m$ up to $16$ in Figure~\ref{fig:Graph} (the gray dots correspond to the limit points $(h(W(X_{(m)})), \frac{1}{m+1})$, the rightmost gray point corresponding to $m = 1$).

\section*{Acknowledgements}

We would like to thank the anonymous referee for suggesting the study of minimal subshifts (Section~\ref{sec:Minimals}), for recommending that we replace our experimental results with the rigorous analysis of Section~\ref{sec:Numerical}, and for supplying us with the asymptotic bound for the maximal root of $x^{m+1} - x^m - k$. We also thank Kaisa Matom\"aki for supplying us with the bound for the maximal root of $x^{m+1} - x^m - 1$, together with the corresponding calculations.

\bibliographystyle{plain}
\bibliography{../../../bib/bib}{}

\end{document}